\documentclass [12pt]{article}
\usepackage{graphicx,amssymb,amsfonts,latexsym,amsmath,amsthm,times}
\usepackage{epsfig}
\usepackage{fancyhdr}
\usepackage{color}
\usepackage[english]{babel}
\numberwithin{equation}{section}
\usepackage{cite}
\newtheorem{Pa}{Paper}[section]
\newtheorem{theorem}[Pa]{Theorem}

\newtheorem{definition}[Pa]{Definition}

\newtheorem{notation}[Pa]{Notation}

\newtheorem{proposition}[Pa]{Proposition}
\newtheorem{remark}[Pa]{Remark}

\def\Om{\Omega}

\def\vp{\varphi}
\def\ve{\varepsilon}

\def\wh{\widehat}
\def\wt{\widetilde}
\def\ov{\overline}
\def\p{\partial}

\def\BC{{\mathbb C}}
\def\BR{{\mathbb R}}

\def\clp{{\mathcal P}}

\def\cln{{\mathcal N}}

\newcommand{\E}{\mathrm{e}}
\newcommand{\I}{\mathrm{i}}

\def\nn{\nonumber}

\newcommand{\cD}{{\mathcal D}}

\newcommand{\cV}{{\mathcal V}}

\title{Weyl functions and the boundary value problem for  a matrix nonlinear Schr\"odinger equation on a
semi-strip}
\author{Alexander Sakhnovich}
\date{}
\begin{document}
\maketitle


\medskip

\thispagestyle{plain}

\vspace{0.5em}

\begin{abstract}  Rectangular matrix solutions of the defocusing nonlinear Schr\"odinger equation (dNLS)
are considered on a semi-strip.  Evolution of the corresponding Weyl function is described in terms of the
initial-boundary conditions. Then initial conditions are recovered from the boundary conditions.
Thus, solutions of  dNLS are recovered from the boundary conditions.

\end{abstract}



\vspace*{1cm}


\section{Introduction} \label{Intro}
\setcounter{equation}{0}
The well-known matrix defocusing nonlinear Schr\"odinger (defocusing NLS or dNLS) equation
\begin{align}  \label{1.1} &
2v_t=\I(v_{xx}-2vv^*v) \qquad \Big(v_t:=\frac{\p}{\p t}v\Big)
\end{align}
is equivalent to the compatibility condition
\begin{align}  \label{zc} &
G_t-F_x+[G,F]=0 \qquad ([G,F]:=GF-FG)
\end{align}
of the auxiliary linear systems
\begin{align}  \label{au} &
y_x=Gy, \qquad y_t=Fy,
\end{align}
where
\begin{align}  \label{1.2} &
G=\I (zj+jV), \quad
F=-\I\big(z^2 j+ z jV-\big(\I V_x-jV^2\big)/2\big), \\
\label{1.3}&
j = \left[
\begin{array}{cc}
I_{m_1} & 0 \\ 0 & -I_{m_2}
\end{array}
\right], \quad  V= \left[\begin{array}{cc}
0&v\\v^{*}&0\end{array}\right],
\end{align}
$I_{m_1}$ is the $m_1\times m_1$ identity matrix and $v$ is an $m_1\times m_2$ matrix function.
We will consider dNLS equation on the semi-strip
\begin{equation}      \label{1.4}
{\mathcal D}=\{(x,\, t):\,0 \leq x <\infty, \,\,0\leq t<a\},
\end{equation}
and we note that the auxiliary system
\begin{align}  \label{1.5} &
y_x=Gy=\I(zj+jV)y
\end{align}
is (for each fixed $t$) a well-known self-adjoint Dirac system, also called AKNS or Zakharov-Shabat system.
The matrix function $v$ is called the potential of the system. Without changes in notations we speak about  usual derivatives inside domains
and about left or right 
(which should be clear from the context) derivatives
on  the boundaries, and boundaries of $\cD$ in particular.

Cauchy problems for nonlinear integrable equations were comparatively thoroughly studied
using the famous Inverse Scattering Transform. The situation with the initial-boundary value problems
for  integrable equations  is much more complicated. The theory of initial-boundary value problems, and their well-posedness in particular, 
is difficult even for the case of linear differential equations, see some results, discussions and references on this topic 
(related also to the integrable nonlinear equations) in \cite{Ash, Bre, DeMaS, Fok, SaAF}.
Boundary value problems for nonlinear equations are often called {\it forced} or {\it perturbed} 
problems. Such problems (as well as some other close to integrable cases) have numerous applications and are of great mathematical interest
(see, e.g., \cite{Persp, Ka}).
Therefore, they are actively investigated and many interesting results are obtained in spite of
remaining difficulties and open problems. First publications on initial-boundary value problems appeared only several years after
the great breakthrough for Cauchy problems for nonlinear integrable equations (see, e.g., \cite{Kv, KaNe}).
Interesting numerical \cite{ChuX}, uniqueness \cite{BW83, Ton} and local existence \cite{Kri, Holm} results followed.
Special {\it linearizable} cases of  boundary conditions were found using symmetrical reduction \cite{Skl} or BT (B\"acklund transformation) method 
\cite{BiTa, Hab}. Global existence results for cubic  NLS equations, Dirichlet and Neumann initial-boundary value problems,
were obtained using PDE methods in \cite{CB} and \cite{Kai}, respectively. Interesting approaches were developed by
D.J. Kaup and H. Steudel \cite{KaSt} and by P. Sabatier (elbow scattering) \cite{Sab1, Sab3}.
Finally, we should mention the well-known {\it global relation method} by A.S. Fokas, see \cite{Fok} and references therein.
(See also some discussions on the corresponding difficulties and open problems in \cite{Ash, BoFo, Fok}.)

In this paper we use the Inverse Spectral Transform approach \cite{Ber, BerG, Kv}.
More precisely, we follow the scheme introduced in \cite{Sa87, Sa88, SaLev2}, see also \cite[Ch. 12]{Sa99}
and references therein. That is, we describe the evolution of Weyl function in terms of linear-fractional transformations.
The scheme is applicable to various  integrable equations \cite{SaA021, SaAg, SaA7, SaAF, SaASG}
and several interesting  uniqueness and existence theorems were proved in this way (see \cite[Ch. 6]{SaSaR}
and references therein for more details). Most of the mentioned above uniqueness and existence theorems 
were obtained for the equations with scalar solutions. Here we consider dNLS \eqref{1.1} and auxiliary system
\eqref{1.5} for the case when $v$ is an $m_1 \times m_2$ matrix function, which is essential, for
instance, for the vector or multicomponent dNLS equations \cite[Ch. 4]{Abl0}. Evolution of the Weyl function of system \eqref{1.5} is described in terms of the
initial-boundary conditions. The connections between initial and boundary conditions are of interest in the theory of PDEs.
In our case the initial condition is recovered from the boundary conditions:
\begin{align}& \label{B3}
V(0,t)=\cV_0(t), \quad V_x(0,t)=\cV_1(t).
\end{align} 
Thus, a new boundary problem is solved and solutions of  dNLS are recovered from the boundary conditions. 

In Section \ref{Evol} we study the evolution of the Weyl function, Section \ref{Quasi} is dedicated to the recovery
of the quasi-analytic initial condition from the boundary conditions on a finite interval, and in Section \ref{quart}
we recover the initial condition, not necessarily quasi-analytic, from the boundary conditions on a semi-axis
(i.e., we consider the case of a quarter-plane).

As usual,  $\BR$ stands for the real axis,
$\BR_+=(0, \, \infty)$,
$\BC$ stands for the complex plain, and
$\BC_+$ for the open upper semi-plane. 
We say that $v(x)$ is locally summable if its entries are summable on all finite intervals of $[0, \, \infty)$.
We say that $v$ is continuously differentiable
if $v$ is differentiable and its first derivatives are continuous.
The notation $\| \cdot \|$ stands for the $l^2$  vector norm or the induced matrix norm.
\section{Evolution of the Weyl function} \label{Evol}
\setcounter{equation}{0}
The main statement in this section is Evolution Theorem \ref{evol}. Its proof is based
on the results of papers \cite{FKRSp1, SaAF}. Some of  those results we formulate here (see also \cite{SaSaR} for greater details and  historical 
remarks).
We denote by $u$ the fundamental solution of system \eqref{1.5} normalized by
the condition
\begin{align}&      \label{nc}
u(0,z)=I_m, \qquad m=m_1+m_2.
\end{align}
\begin{definition} \label{CyW2} Let Dirac system \eqref{1.5} on $[0, \, \infty)$
 be given and assume that $v$ is  locally summable.
Then Weyl function is an $m_2 \times m_1$ holomorphic matrix function, which satisfies the inequality
\begin{align}&      \label{2.0}
\int_0^{\infty}
\begin{bmatrix}
I_{m_1} & \vp(z)^*
\end{bmatrix}
u(x,z)^*u(x,z)
\begin{bmatrix}
I_{m_1} \\ \vp(z)
\end{bmatrix}dx< \infty .
\end{align}
\end{definition}
The following proposition is proved in \cite{FKRSp1} (and in  \cite[Section 2.2]{SaSaR}).
\begin{proposition} The Weyl function always exists and it is unique.
\end{proposition}
In order to construct the Weyl function, we introduce a class  of nonsingular $m \times m_1$ matrix functions 
$\clp(z)$ with property-$j$, which are an immediate analog of the classical pairs
of parameter matrix functions. Namely, the matrix functions 
$\clp(z)$ are meromorphic in $\BC_+$ and satisfy
(excluding, possibly, a discrete set of points)
the following relations
\begin{align}\label{2.1}&
\clp(z)^*\clp(z) >0, \quad \clp(z)^* j \clp(z) \geq 0 \quad (z\in \BC_+).
\end{align}
Relations \eqref{2.1} imply (see, e.g., \cite{FKRSp1}) that
\begin{align}&      \label{wd}
\det \Big(\begin{bmatrix}
I_{m_1} & 0
\end{bmatrix}u(x,z)^{-1}\clp(z)\Big)\not= 0.
\end{align}
\begin{definition} \label{set}
The set $\cln(x,z)$ of M\"obius transformations is the set of values $($at the fixed $x\in [0, \, \infty), \,z\in \BC_+)$ 
of matrix functions
\begin{align}\label{2.2}&
\vp(x,z,\clp)=\begin{bmatrix}
0 &I_{m_2}
\end{bmatrix}u(x,z)^{-1}\clp(z)\Big(\begin{bmatrix}
I_{m_1} & 0
\end{bmatrix}u(x,z)^{-1}\clp(z)\Big)^{-1},
\end{align}
where $\clp(z)$ are nonsingular  matrix functions 
 with property-$j$. 
 \end{definition}
 \begin{remark}\label{MB}
 It was shown in \cite{FKRSp1} that a family $\cln(x,z)$, where $x$ increases to infinity and $z \in \BC_+$ is fixed,
is a family of embedded matrix balls such that the right semi-radii
 are uniformly bounded and the left semi-radii tend to zero.
 \end{remark}
 \begin{proposition}\cite{FKRSp1} \label{PnW1} Let Dirac system \eqref{1.5} on $[0, \, \infty)$
 be given and assume that $v$ is  locally summable.
 Then the sets $\cln(x,z)$
 are well-defined. There is a unique matrix function
 $\vp(z)$ in $\BC_+$ such that
\begin{align}&      \label{2.3}
\{\vp(z) \}=\bigcap_{x<\infty}\cln(x,z).
\end{align} 
This function is analytic and non-expansive.
Furthermore, this function coincides with the Weyl function of system \eqref{1.5}.
 \end{proposition}
Formula \eqref{2.3} is completed by the asymptotic relation
\begin{align}&      \label{2.3!}
\vp(z)=\lim_{b \to \infty}\vp_b(z),
\end{align} 
which is valid for any set of functions $\vp_b(z)\in \cln(b,z)$. Relation \eqref{2.3!} follows from \eqref{2.3}
and Remark \ref{MB} (see also \cite[Remark 2.24]{SaSaR}).

Next, we consider the famous compatibility condition (zero curvature equation) \eqref{zc}.
The sufficiency of this condition  was studied in a more rigorous way and 
the corresponding important factorization formula for fundamental solutions was
introduced in \cite{Sa87, Sa88, Sa99}. The  factorization formula for fundamental solutions
(see \eqref{2.6} below)
was proved in greater detail and under weaker
conditions in \cite{SaAF}. More specifically, we have the following general proposition,
where specific form \eqref{1.2} of $G$ and $F$ is not essential. 
\begin{proposition} \label{TmM} \cite{SaAF}
Let some $m \times m$ matrix functions $G$ and $F$ and their derivatives $G_t$ and
$F_x$  exist on the semi-strip $\cD$,
let $G$, $G_t$ and $F$ be continuous with respect to $x$ and $t$ on $\mathcal{D}$,
 and let \eqref{zc}
hold. Then we have the equality 
\begin{equation} \label{2.4}
u(x,t,z)R(t,z)=R(x,t,z)u(x,0,z), \quad R(t,z):=R(0,t,z),
\end{equation}
where $u(x,t,z)$ and $R(x,t,z)$ are normalized fundamental solutions given, respectively, by:
\begin{equation} \label{2.5}
u_x=Gu, \quad u(0,t,z)=I_m; \quad R_t=FR, \quad R(x,0,z)=I_m.
\end{equation}
The equality \eqref{2.4} means that the matrix function 
$$y(x,t,z)=u(x,t,z)R(t,z)=R(x,t,z)u(x,0,z)$$ 
satisfies both systems \eqref{au} in 
$\cD$. Moreover, the fundamental solution $u$ admits the factorization
\begin{equation} \label{2.6}
u(x,t,z)=R(x,t,z)u(x,0,z)R(t,z)^{-1}.
\end{equation}
\end{proposition}
In a way, which is similar to the $m_1=m_2$ case (see \cite{Sa88})
and to the focusing NLS equation case (see \cite{FKRSskew, SaA021})
we obtain our main result in this section.
  \begin{theorem}\label{evol} Let an $m_1 \times m_2$ matrix function $v(x,t)$ 
 be  continuously differentiable on $\cD$
 and let  $v_{xx}$ exist. 
 Assume that $v$  satisfies the dNLS equation \eqref{1.1}  as well as  the following inequalities
 $($for all $0 \leq t<a$ and some values $M(t)\in \BR_+):$
  \begin{equation} \label{2.7}
  \sup_{x \in \BR_+, \, 0\leq s \leq t}\|v(x,s)\| \leq M(t).
\end{equation}
 Then the evolution $\vp(t,z)$ of the Weyl functions  of  Dirac systems  \eqref{1.5} is given $($for $\Im(z)>0)$ by the equality
 \begin{equation} \label{2.8}
\vp(t,z)=\big(R_{21}(t,z)+R_{22}(t,z)\vp(0,z)\big)
\big(R_{11}(t,z)+R_{12}(t,z)\vp(0,z)\big)^{-1}.
\end{equation}
\end{theorem}
\begin{proof} Using the equivalence between the dNLS and  the compatibility condition \eqref{zc}
for  $G$ and $F$ given
by \eqref{1.2},  and taking into account
 the smoothness conditions on $v$, we see that  the requirements of Proposition \ref{TmM} are satisfied,
that is, \eqref{2.6} holds.  We rewrite \eqref{2.6} in the form
\begin{equation} \label{2.6'}
u(x,t,z)^{-1}=R(t,z)u(x,0,z)^{-1}R(x,t,z)^{-1}.
\end{equation}
First, we consider $R(x,t,z)$.
Since $V=V^*$ and $jV =-V j$, it follows from \eqref{1.2}
  that  
  $$F(x,s, z)^*j+jF(x,s,  z)=
 \I(\ov z-z)\big((z+\ov z)I_m+V(x,s) \big).
  $$ 
  Hence, using    \eqref{2.5} we get
\begin{align}&\nn
  \frac{\p}{\p s}\big(R(x,s,  z)^*jR(x,s,z)\big)=R(x,s,  z)^*(F(x,s,z)^*j+jF(x,s,  z))R(x,s,z)
  \\ & \label{2.9}=
\I(\ov z-z)R(x,s,  z)^*\big((z+\ov z)I_m+V(x,s) \big)R(x,s,z) .
\end{align}
In view of \eqref{2.7} and \eqref{2.9} the inequality
\begin{align}&\label{2.10}
  \frac{\p}{\p s}\big(R(x,s,  z)^*jR(x,s,z)\big)<0 \quad (0\leq s \leq t) \end{align}
  holds in  the quarterplane
\begin{align}&\label{2.11}
\Om_t=\{z \in \BC_+:  \Re z<-M(t)/2\}.
 \end{align}
 Inequality \eqref{2.10} and the initial condition
 $R(x,0,z)=I_m$ imply $R(x,t,  z)^*jR(x,t,z)\leq j$, or, equivalently,
\begin{align} \label{2.12}&
\big(R(x,t,  z)^*\big)^{-1}jR(x,t,z)^{-1}\geq j, \qquad z \in \Om_t.
\end{align}

For  ${\cal P}(z)$ satisfying \eqref{2.1}, we  determine $\wt\clp(x,t,z)$ (sometimes we 
 write also $\wt\clp(z)$, omitting $x$ and $t$) by the equality
\begin{equation} \label{2.13}
\wt\clp(x,t,z):=R(x,t,z)^{-1}\clp(z).
\end{equation}
In view of  \eqref{2.12}  the matrix function $\wt\clp$ satisfies \eqref{2.1} 
in $\Om_t$.
Using \eqref{2.6'} and \eqref{2.13}, we see that
\begin{align} \label{2.14}
u(x,t,z)^{-1}\clp(z)=
R(t,z)
\begin{bmatrix} I_{m_1} \\ \phi(x,t, z) \end{bmatrix}
\begin{bmatrix} I_{m_1} & 0 \end{bmatrix}u(x,0,z)^{-1}\wt \clp(z),
\end{align}
where $\wt \clp(z)= \wt \clp(x,t,z)$ and
\begin{align} \label{2.15}
\phi(x,t,z)=
\begin{bmatrix} 0 & I_{m_2} \end{bmatrix}u(x,0,z)^{-1}\wt \clp(x,t,z)
\left(\begin{bmatrix} I_{m_1} & 0 \end{bmatrix}u(x,0,z)^{-1}\wt \clp(x,t,z)\right)^{-1}.
\end{align}
Here, \eqref{wd} yields the following inequalities in $\Om_t$:
\begin{align} \label{2.16}
\det \left(\begin{bmatrix} I_{m_1} & 0 \end{bmatrix}u(x,0,z)^{-1}\wt \clp(z)\right)\not=0,
\quad
\det \left(\begin{bmatrix} I_{m_1} & 0 \end{bmatrix}u(x,t,z)^{-1} \clp(z)\right)\not=0.
\end{align}
According to \eqref{2.14} and \eqref{2.16} we have (in the quarterplane $\Om_t$) the equality
\begin{align}\nn&
\begin{bmatrix}
0 &I_{m_2}
\end{bmatrix}u(x,t,z)^{-1}\clp(z)\Big(\begin{bmatrix}
I_{m_1} & 0
\end{bmatrix}u(x,t,z)^{-1}\clp(z)\Big)^{-1}
\\ & \label{2.17}
=\big(R_{21}(t,z)+R_{22}(t,z)\phi(x,t,z)\big)
\big(R_{11}(t,z)+R_{12}(t,z)\phi(x,t,z)\big)^{-1},
\end{align}
where
\begin{align}& \label{2.18}
\det\big(R_{11}(t,z)+R_{12}(t,z)\phi(x,t,z)\big)\not=0.
\end{align}
Next, we recall that according to Proposition \ref{PnW1} the inequality
\begin{align}& \label{2.18'}
\begin{bmatrix} I_{m_1} & \vp(t,z)^*\end{bmatrix} j \begin{bmatrix} I_{m_1} \\ \vp(t,z)\end{bmatrix}\geq 0
\end{align}
holds for the Weyl functions $\vp$. Using \eqref{2.9}, in a way similar to the proof of \eqref{2.12}
we show that the inequality
\begin{align}&\label{2.19}
 R(t,  z)^*jR(t,z) \geq j 
 \end{align}
  holds in  the quarterplane
\begin{align}&\label{2.20}
\wh \Om_t=\{z \in \BC_+:  \Re z>M_0(t)/2\}, \quad M_0(t):=\max_{r\in [0, \, t]}\|v(0,r)\|.
 \end{align}
Because of  \eqref{2.18'} and \eqref{2.19} we have the inequality
\begin{align}& \label{2.21}
\det\big(R_{11}(t,z)+R_{12}(t,z)\vp(0,z)\big)\not=0
\end{align}
in $\wh \Om_t$. Hence, in view of the analyticity, the inequality \eqref{2.21} holds in $\BC_+$ excluding,
possibly, some discrete points.

Taking into account Definition \ref{set}, we see that the left-hand side
of \eqref{2.17} belongs $\cln(x,t,z)$ and $\phi(x,t,z) \in \cln(x,0,z)$,
where $\phi$ is defined in \eqref{2.15}. Therefore, \eqref{2.8} follows from
\eqref{2.3!},  \eqref{2.17}, and \eqref{2.21}, when $x$ tends to infinity (and $z\in \Om_t$). Although, we first derived
\eqref{2.8} only for $z \in \Om_t$, we see that it holds everywhere in $\BC_+$ via analyticity. 
\end{proof}
\begin{remark}\label{RkInvPr}  According to \cite{Sa14}, the Dirac system $y_x=Gy$ $($where $G$ is given by \eqref{1.2} and \eqref{1.3}
and $v$ is locally square summable$)$
is uniquely recovered from the Weyl function $\vp$. Thus, $v$ is uniquely recovered from $\vp$, see the procedure in \cite[Theorem 4.4]{Sa14}.
The case of a more smooth (i.e., locally bounded) $v$ was dealt with in \cite{SaSaR}, see also references therein.
\end{remark}
\section{Quasi-analytic initial condition}\label{Quasi}
\setcounter{equation}{0}
The class $C\big(\{\wt M_k\}\big)$  
consists of all infinitely differentiable on $[0, \infty)$ scalar functions $f$ such that for some 
$a(f) \geq 0$ and  for fixed constants $\wt M_k >0 $ $(k \geq 0)$ we have
\begin{align}& \label{B1}
\left| \frac{d^k f}{dx^k}(x)\right|\leq a(f)^{k+1}\wt M_k \quad {\mathrm{for \, all}} \quad x\in [0, \infty).
\end{align}
Here, we use the notation $\wt M_k$ (as well $\wt M$ below)  because the upper estimates $M$  (without tilde) were already used
in the previous section.
Recall that $C\big(\{\wt M_k\}\big)$    is called quasi-analytic if 
for the functions $f$ from this class and for any $ x\geq 0$ the equalities $\frac{d^k f}{dx^k}(x)=0$ $(k\geq 0)$ yield $f\equiv 0$.
According to the famous Denjoy--Carleman theorem, the inequality
\begin{align}& \label{B2}
\sum_{n=1}^{\infty}\frac{1}{L_n}=\infty, \quad L_n:=\inf_{k\geq n}\wt M_k^{1/k}
\end{align}
implies that  the class $C\big(\{\wt M_k\}\big)$ is quasi-analytic.
\begin{notation}\label{NnB1} We consider  $m_1 \times m_2$ matrix functions $v(x,t)$,
which are continuously  differentiable
on the semistrip $\cD$, satisfy \eqref{2.7} and are 
such that $v_{xx}$ exists  and
the entries $v_{ij}(x,0)$ of  $v(x,0)$ belong to the quasi-analytic classes $C\big(\{\wt M_k(i,j)\}\big)$.
Moreover, we require that for each $k$ there is a value $\ve_k=\ve_k(v)>0$ such that $v$
is $k$ times continuously differentiable with respect to $x$ in the square 
\begin{align}& \label{B2'}
\cD({\ve_k})=\{(x,\, t): \quad  0\leq x\leq \ve_k, \quad 0\leq t\leq \ve_k\}, \quad \cD({\ve_k})\subset \cD.
\end{align}
The classes of such functions $v(x,t)$ are denoted by $C_q(\cD, \wt M)$, where 
$\wt M=\{M_k(i,j)\}$.
\end{notation}
Further we shall need a stronger version of the well-known Clairaut's (or Schwarz's)  theorem on mixed derivatives
(see, e.g., \cite{Aks, Mar, See}). Since we need this version in the closed rectangular
$$\cD(\ve, \wh \ve)=\{(x, \, t): \quad 0\leq x\leq \ve, \quad 0\leq t\leq \wh \ve\},$$
we note that  a simple proof of Clairaut's theorem, which is given, for instance, in \cite{See}, is valid for the
whole domain $\cD(\ve, \wh \ve)$, and not only for the interior of $\cD(\ve, \wh \ve)$ as formulated in the statement
(C) from \cite{See}. (In fact, we should switch sometimes from the integrals $\int_0^{x}\cdot ds$
and $\int_0^{t}\cdot ds$ in this proof to the integrals $\int_x^{\ve}\cdot ds$
and $\int_t^{\wh \ve}\cdot ds$, respectively.) Thus, we have the following proposition.
\begin{proposition}\label{PnB2} If the functions $f$, $f_t$ and $f_{tx}$ exist and are continuous on $\cD(\ve, \wh \ve)$
and the derivative $f_x(x, 0)$ exists for $0\leq x \leq  \ve$, then $f_x$ and $f_{xt}$ exist on $\cD(\ve, \wh \ve)$
and $f_{xt}=f_{tx}$. $($Here $f_{xt} = \p f_x / \p t$ and $f_{tx} = \p f_t / \p x.)$
\end{proposition}
Now, using Theorem \ref{evol}, Remark \ref{RkInvPr} and Notation \ref{NnB1} we can formulate and prove
a uniqueness theorem. 
\begin{theorem} \label{recqua} Assume that  $v\in C_q(\cD, \wt M)$ satisfies on $\cD$ the dNLS equation \eqref{1.1}.
Then this $v(x,t)$ is uniquely determined $($inside the class $C_q(\cD, \wt M))$ by the boundary conditions \eqref{B3}.
\end{theorem}
\begin{proof}  It is immediate from Notation \ref{NnB1} and the definition of quasi-analytic classes that $v(x,0)$
is uniquely determined by the derivatives $\left(\frac{d^k}{dx^k}v \right)(0,0)$ $(k \geq 0)$.
Let us recover recursively these derivatives from the boundary conditions.

First, we note that $v_t(0,t)=\frac{d}{dt}\cV_0(t)$ and that $v_{xx}(0,t)$  is immediately recovered from $v(0,t)$, $v_t(0,t)$ and
the dNLS equation  \eqref{1.1}.
Without loss of generality, we assume that the values $\ve_k$ in \eqref{B2'} monotonically decrease.
Since $v$ is three times continuously differentiable (with respect to $x$) in $\cD({\ve_3})$
and dNLS holds, we see that $v_t$ and $v_{tx}$ exist and are continuous. Here, in order to show
that $v_{tx}$ satisfies the conditions above, we differentiate both sides of dNLS: 
\begin{align}& \label{B4}
2v_{tx}=\I\left(v_{xxx}-2\frac{\p}{\p x}(vv^*v)\right).
\end{align}
Hence, in view of
Proposition \ref{PnB2}, we have $v_{xt}=v_{tx}$. We see that $v_{xt}(0,t)$ and so $v_{tx}(0,t)$ is recovered from the second boundary
condition in \eqref{B3}. 

Next, we assume that the functions $\left(\frac{\p^k}{\p x^k}v_t\right)(0,t)$, where  $0 \leq t \leq \ve_{r+2}$, are already recovered for $0 \leq k \leq r$,
 the functions $\left(\frac{\p^k}{\p x^k}v\right)(0,t)$ are recovered for $\, 0 \leq k \leq r+1$, that
\begin{align}& \label{B5}
\left(\frac{\p^k}{\p x^k}v\right)_{xt}= \left(\frac{\p^k}{\p x^k}v\right)_{tx} \quad {\mathrm{for}} \quad 0 \leq k \leq r-1, \quad (x,t) \in \cD(\ve_{r+2}),
\end{align}
and that both sides of the equalities in \eqref{B5} are continuous. Differentiating $r$ times (with respect to $x$) both sides of dNLS 
and rewriting the result in the form
\begin{align}& \label{B6}
\frac{\p^{r+2}}{\p x^{r+2}}v=2\frac{\p^r}{\p x^r}(vv^*v)-2\I \frac{\p^r}{\p x^r}v_{t},
\end{align}
we recover $\left(\frac{\p^{r+2}}{\p x^{r+2}}v\right)(0,t)$ from the already recovered derivatives. 
Differentiating $r+1$ times both sides of dNLS, we derive that the derivative $\frac{\p^{r+1}}{\p x^{r+1}}v_t$ exists and is continuous
in $\cD(\ve_{r+3})$. Taking into account \eqref{B5}, we obtain that 
$$\frac{\p^{r+1}}{\p x^{r+1}}v_t= \left(\frac{\p^{r}}{\p x^{r}}v\right)_{tx}.$$
Hence, the right-hand side of the equality above is continuous and the conditions of Proposition \ref{PnB2}
hold for $f=\frac{\p^{r}}{\p x^{r}}v$. Thus, the equality in \eqref{B5} is valid for $k=r$, that is,
\begin{align}& \label{B7}
\left(\frac{\p^k}{\p x^k}v\right)_{xt}= \left(\frac{\p^k}{\p x^k}v\right)_{tx} \quad {\mathrm{for}} \quad 0 \leq k \leq r, \quad (x,t) \in \cD(\ve_{r+3}).
\end{align}
Differentiating $\left(\frac{\p^{r+1}}{\p x^{r+1}}v\right)(0,t)$ with respect to $t$ and taking into account  \eqref{B7}, we recover
$$\left(\frac{\p^{r+1}}{\p x^{r+1}}v_t\right)(0,t)=\frac{\p}{\p t}\left(\left(\frac{\p^{r+1}}{\p x^{r+1}}v\right)(0,t)\right).$$
Thus, we showed by induction that our assumptions hold for all $r\geq 1$ and, in particular,
the derivatives $\left(\frac{d^k}{dx^k}v \right)(0,0)$ $(k \geq 0)$ are uniquely recovered from the boundary
conditions. In other words, $v(x,0)$ is uniquely determined by the boundary
conditions \eqref{B3}.

In turn, the potential $v(x,0)$ uniquely determines the Weyl function $\vp(0,z)$ (see Proposition \ref{PnW1}). Then, the
evolution $\vp(t,z)$ of the Weyl function is described in Theorem \ref{evol}. Finally, according to Remark \ref{RkInvPr}, 
the matrix function $v(x,t)$ is uniquely recovered from $\vp(t,z)$.
\end{proof}
\begin{remark}\label{Rec} We see that the scheme to recover $v$ from the boundary conditions
follows from Proposition \ref{PnW1}, Theorem \ref{evol} and the proof of Theorem \ref{recqua}.
The only step that we did not describe in detail is the recovery of $v(x,0)$ from the Taylor coefficients
$\left(\frac{d^k}{dx^k}v \right)(0,0)$. Although Taylor coefficients  uniquely determine $v(x,0)$,
the recovery of this function presents an interesting problem, which is not solved completely so far.  
See  \cite{Bang} and \cite[Section III.8]{Beur}  for some important results.
\end{remark}

\section{Defocusing NLS in a quarter-plane}\label{quart}
\setcounter{equation}{0}
Let us consider dNLS for the case that $a=\infty$ in \eqref{1.4}, that is, consider dNLS in the quarter-plane
\begin{equation}      \label{B8}
{\mathcal D}(\infty)=\{(x,\, t):\,0 \leq x <\infty, \,\,0\leq t<\infty\}.
\end{equation}
Then, assuming that $\cV_0(t)$ and $\cV_1(t)$ are bounded, that is, the inequalities
\begin{align}& \label{B9}
\sup_{0\leq t<\infty} \| \cV_0(t)\|<\wh M, \quad \sup_{0\leq t<\infty} \| \cV_1(t)\|<\breve M
\end{align}
hold for some $\wh M, \breve M \, > \, 0$, we will express the Weyl function $\vp(0,z)$ via the boundary conditions \eqref{B3}.
Taking into account Theorem \ref{evol} and Remark \ref{RkInvPr},  we see that in this way we express
the solution $v(x,t)$ of dNLS via boundary conditions. For simplicity, we assume also that both $\cV_0(t)$ and $\cV_1(t)$ are
continuous and note that, in view of \eqref{B9}, $R$ has the following property.
\begin{remark}\label{Rkbd} For a fixed $z$ and each $\delta>0$, there is a value $\ve(\delta)$ such that
if $\|R(s,z)f\|\geq \delta$, then $\|R(t,z)f\|>\delta /2$ for all $t$ from the interval $0 \leq t-s \leq \ve$.
Indeed, from the multiplicative integral representation of the solutions of differential equations $($similar to the considerations in \cite[p.~191]{SaSaR}$)$,  we obtain
\begin{align}& \label{B17}
R(t,z)R(s,z)^{-1}=\lim_{N \to \infty}\prod_{k=1}^N\E^{\frac{\Delta}{N}F_k(N)}, \quad \Delta=t-s, \quad F_k(N)=F(0, s+\frac{k\Delta}{N},z).
\end{align}
Hence, we derive
\begin{align}\nn
\| R(t,z)R(s,z)^{-1}-I_m\| & \leq \overline{\lim_{N \to \infty}}\prod_{k=1}^N\E^{\frac{\Delta}{N}C}-1\leq 
\overline{\lim_{N \to \infty}}\prod_{k=1}^N\Big(1+\frac{\Delta}{N}C\E^{\frac{\Delta}{N}C}\Big)-1
\\ &  \label{B18}
\leq
\lim_{N \to \infty}\Big(1+\frac{\Delta }{N}C_1\Big)^N-1=\E^{C_1\Delta }-1
\end{align}
for some constants $C$, $C_1$ $(C\geq \|F\|)$. According to \eqref{B18}, when $\|R(s,z)f\|\geq \delta$  and $\Delta$ is small enough,
we have $\|R(t,z)f\|>\delta /2$.
\end{remark}
\begin{theorem}\label{TmQP} Let $v$  satisfy the conditions of Theorem    \ref{evol} 
and continuous boundary conditions \eqref{B3}, for which \eqref{B9} holds. Then, in the domain
\begin{align}& \label{B10}
\Om=\{z: \quad \Im(z) \geq 1/2, \quad \Re(z)\leq - \wh M\}
\end{align}
we have the equality
\begin{align}& \label{B11}
\vp(0,z) = -\lim_{t \to \infty} R_{22}(t,z)^{-1}R_{21}(t,z),
\end{align}
where $R_{ik}(t,z)$ are the blocks of $R(t,z)$, which is given by the relations
$R_t=FR$, $R(0,z)=I_m$, and $V$ and $V_x$ in $F$ may be substituted by their boundary
values $\cV_0(t)$ and $\cV_1(t)$, respectively.
\end{theorem}
\begin{proof}  Setting $x=0$ in \eqref{2.9}  and taking into account the first inequality in \eqref{B9}, we obtain
\begin{align}& \label{B12}
\frac{d}{dt}\big(-R(t,z)^*jR(t,z)\big) \geq \wh M R(t,z)^*R(t,z), \quad z\in \Om.
\end{align}
On the other hand, formula  \eqref{2.8} for Weyl functions $\vp(t,z)$ yields
\begin{align}& \label{B13}
\begin{bmatrix} I_{m_1} \\ \vp(t,z)
\end{bmatrix}=R(t,z)
\begin{bmatrix} I_{m_1} \\ \vp(0,z)
\end{bmatrix}
\big(R_{11}(t,z)+R_{12}(t,z)\vp(0,z)\big)^{-1}.
\end{align}
Relations \eqref{2.18'} and \eqref{B13} imply that
\begin{align}& \label{B14}
\begin{bmatrix} I_{m_1} & \vp(0,z)^*
\end{bmatrix}R(t,z)^*jR(t,z)
\begin{bmatrix} I_{m_1} \\ \vp(0,z)
\end{bmatrix}
\geq 0.
\end{align}
Since $R(0,z)=I_m$, we derive from \eqref{B12} and \eqref{B14}
the inequality
\begin{align}& \nn
\begin{bmatrix} I_{m_1} & \vp(0,z)^*
\end{bmatrix}\int_0^{t}R(s,z)^*R(s,z)ds 
\begin{bmatrix} I_{m_1} \\ \vp(0,z)
\end{bmatrix} \leq
(1/\wh M)\begin{bmatrix} I_{m_1} & \vp(0,z)^*
\end{bmatrix} j
\begin{bmatrix} I_{m_1} \\ \vp(0,z)
\end{bmatrix}
, \quad {\mathrm{i.e.}},
\\ & \label{B16}
\int_0^{\infty}\begin{bmatrix} I_{m_1} & \vp(0,z)^*
\end{bmatrix}R(s,z)^*R(s,z)
\begin{bmatrix} I_{m_1} \\ \vp(0,z)
\end{bmatrix}ds
\leq (1/\wh M) I_{m_1}.
\end{align}
From \eqref{B16} and Remark \ref{Rkbd} we see that
\begin{align}& \label{B19}
\lim_{t\to \infty}\left\| R(t,z)
\begin{bmatrix} I_{m_1} \\ \vp(0,z)
\end{bmatrix}
\right\|=0.
\end{align}
It is immediate from \eqref{B12} that for $z\in \Om$ we have
\begin{align}\nn
R_{22}(t,z)^*R_{22}(t,z) &\geq R_{22}(t,z)^*R_{22}(t,z)-R_{12}(t,z)^*R_{12}(t,z)
\\ & \label{B20}
\geq R_{22}(0,z)^*R_{22}(0,z)-R_{12}(0,z)^*R_{12}(0,z)=I_{m_2}.
\end{align}
 Finally relations \eqref{B19} and \eqref{B20} yield \eqref{B11}.
\end{proof}
We note that for the scalar case this result was derived in \cite{Sa87, SaLev2}.
An analog of this result for the  focusing NLS (the square matrix case) was obtained in
\cite{SaA021}. See also some analogs for the sine-Gordon model in \cite{SaASG} and
\cite[Corollary 6.25]{SaSaR}. 
\section*{Acknowledgements} The research  was
supported by the  Austrian Science Fund (FWF) under Grant  No. P24301.



\end{document}